\documentclass[final,reqno]{amsart}
\usepackage[utf8]{inputenc}
\usepackage{amsmath}
\usepackage{amsthm}
\usepackage{amsfonts}
\usepackage{amssymb}
\usepackage[all]{xy}
\usepackage{indentfirst}
\usepackage[pagebackref=true]{hyperref}

\newtheorem{thm}{Theorem}[section]

\newtheorem{theorem}[thm]{Theorem}
\newtheorem{corollary}[thm]{Corollary}
\newtheorem{proposition}[thm]{Proposition}

\newtheorem{lemma}[thm]{Lemma}
\newtheorem*{theorem*}{Theorem}

\theoremstyle{definition}

\newtheorem{example}[thm]{Example}
\newtheorem{definition}[thm]{Definition}

\newtheorem{remark}[thm]{Remark}

\DeclareMathOperator{\Span}{span}
\DeclareMathOperator{\Int}{Int}

\newcommand{\N}{\mathbb{N}} %% Naturals
\newcommand{\Z}{\mathbb{Z}} %% Integers
\newcommand{\C}{\mathbb{C}} %% Complex
\usepackage{verbatim} %% Enables 'comment'
\usepackage{color}
\newcommand{\LRA}{\longrightarrow}  %% Arrows

\title[Topological full groups of groupoids]{$\mathrm C^*$-simplicity and representations of topological full groups of groupoids}

\author{Kevin Aguyar Brix}
\address[K.A. Brix]{Department of Mathematical Sciences, University of Copenhagen, Universitetsparken 5, DK-2100 Copenhagen, Denmark}
\email{kab@math.ku.dk}

\author{Eduardo Scarparo}
\address[E. Scarparo]{Departamento de Matemática, Universidade Federal de Santa Catarina, 88040-970 Florianópolis-SC, Brazil}
\email{eduardo.scarparo@posgrad.ufsc.br}

\thanks{The first named author is supported by the Danish National Research Foundation through the Centre for Symmetry and Deformation (DNRF92). The second named author is supported by CNPq, Brazil, 167983/2017-2.}

\begin{document}

\begin{abstract}
    Given an ample groupoid $G$ with compact unit space, 
    we study the canonical representation of the topological full group $[[G]]$ in the full groupoid $\mathrm C^*$-algebra $\mathrm C^*(G)$. 
    In particular, we show that the image of this representation generates $\mathrm C^*(G)$ if and only if $\mathrm C^*(G)$ admits no tracial state.
    The techniques that we use include the notion of groups covering groupoids.
    
    As an application, we provide sufficient conditions for $\mathrm C^*$-simplicity of certain topological full groups, 
    including those associated with topologically free and minimal actions of non-amenable and countable groups on the Cantor set.
\end{abstract}
\maketitle

\section{Introduction}

Topological full groups associated to group actions on the Cantor set have given rise to examples of groups with interesting new properties. 
See, e.g., \cite{MR3071509} and ~\cite{nekrashevych2016palindromic} for recent developments.
In the context of groupoids, the topological full group was introduced by H.~Matui in~\cite{MR2876963}, 
who investigated their relation with homology groups of groupoids.

Following a slightly different approach, V.~Nekrashevych (\cite{MR3904185}) defined the topological full group $[[G]]$ 
of an ample groupoid $G$ with compact unit space to consist of the clopen bisections $U\subset G$ such that $r(U)=s(U)=G^{(0)}$.
In this paper, we study the unitary representation $\pi\colon [[G]]\LRA \mathrm C^*(G)$ given by $\pi(U) := 1_U$, for every $U\in[[G]]$.
Let $\mathrm C^*_\pi([[G]])$ denote the $\mathrm C^*$-algebra generated by $\pi([[G]])$ in $\mathrm C^*(G)$.

Our main result is as follows:

\begin{theorem*}[Theorems~\ref{imp} and~\ref{main}]
Let $G$ be an ample groupoid with compact unit space such that the orbit of each $x\in G^{(0)}$ has at least three points. 
Then $\overline{\Span}\{1-1_U\in \mathrm C^*(G) \mid U\in[[G]]\}$ is a hereditary $\mathrm C^*$-subalgebra of $\mathrm C^*(G)$. 
Moreover, $\mathrm C^*(G)$ admits no tracial state if and only if $\mathrm C^*_\pi([[G]])= \mathrm C^*(G)$.
\end{theorem*}

This generalizes part of~\cite[Proposition 5.3]{MR3630641} (see Remark~\ref{ho}). 
If, in addition, $G$ is second countable, essentially principal and minimal, 
then $\mathrm C^*_r(G)$ is stably isomorphic to $\overline{\Span}\{1-1_U\in \mathrm C_r^*(G)\mid U\in[[G]]\}$ (Corollary~\ref{cor:imp-reduced}).

Given an ample groupoid $G$ with compact unit space, let $\pi_r$ denote the canonical representation of $[[G]]$ in $\mathrm C^*_r(G)$.

Recall that a group is said to be $\mathrm C^*$-simple if its reduced $\mathrm C^*$-algebra is simple. 
Recently, there has been a lot of progress in understanding this notion, and new characterizations of $\mathrm C^* $-simplicity have been obtained (see \cite{MR3735864}, \cite{MR3652252}, \cite{2015arXiv150901870K}). In~\cite{boudec2016subgroup}, 
A.~Le Boudec and N.~Matte Bon showed that a countable group of homeomorphisms on a Hausdorff space $X$ is $\mathrm C^*$-simple if the rigid stabilizers of non-empty and open subsets of $X$ are non-amenable. 
By using this result, we show the following:

\begin{theorem*}[Theorem~\ref{thm:nonamenable-C-simple}]
Let $G$ be a second countable, essentially principal, minimal and ample groupoid with compact unit space. 
If 
\begin{enumerate}
\item[(i)]$G$ is not amenable, or
\item[(ii)]$\pi_r$ does not weakly contain the trivial representation,
\end{enumerate}
 then $[[G]]$ is $\mathrm C^*$-simple.
\end{theorem*}

Consequently, the topological full group associated with a topologically free and minimal action of a countable and non-amenable group on the Cantor set is $\mathrm C^*$-simple (Corollary~\ref{cor:action-C-simple}). 
For free actions, this was shown to be true in~\cite{boudec2016subgroup}.

The paper is organized as follows.
In Section~\ref{pre}, we collect basic definitions about groupoids, establish notation and present some relevant examples.

In Section~\ref{cover}, we study groups covering groupoids. 
Given an ample groupoid $G$ with compact unit space, a subgroup $\Gamma\leq [[G]]$ is said to cover $G$ if $G = \bigcup_{U\in\Gamma} U$. 
We investigate under which conditions $[[G]]$ covers $G$ and show that $\mathrm C^*_\pi([[G]])$ admits a character 
if and only if $G^{(0)}$ admits a $G$-invariant probability measure (Corollary~\ref{nice}).

In Section~\ref{mains}, we analyze the representation of the topological full group in the full and the reduced groupoid $\mathrm C^*$-algebras to reach the main theorem above.

In Section~\ref{simp}, we apply the results of Sections~\ref{cover} and~\ref{mains} in order to study $\mathrm C^*$-simplicity of the topological full group.

%%%%%%%%%%%%%%%%%%%%%%%%%%%%%%%%%%%%%%%%%%%%%%%%%%%%%%%%%%%%%%%%%%%%%%%%%%%%%%%%%%%%%%%%%%%%%%%%%%%%%%%%%%
%%%%%%%%%%%%%%%%%%%%%%%%%%%%%%%%%%%%%%%%%%%%%%%%%%%%%%%%%%%%%%%%%%%%%%%%%%%%%%%%%%%%%%%%%%%%%%%%%%%%%%%%%%
\section{Preliminaries}\label{pre}
In this section we introduce relevant concepts and establish notation.
Throughout the paper, we let $\N = \{0,1,2,3\ldots\}$ denote the non-negative integers.

\subsection{Ample groupoids}
A topological groupoid $G$ is \emph{ample} if $G$ is locally compact, Hausdorff, \'etale 
(in the sense that the range and source maps $r,s\colon G\LRA G$ are local homeomorphisms onto $G^{(0)}$) 
and the unit space $G^{(0)}$ is totally disconnected. 
The \emph{orbit} of a point $x\in G^{(0)}$ is the set $G(x) := r(s^{-1}(x))$, and $G$ is said to be \emph{minimal} if $\overline{G(x)}=G^{(0)}$, for every $x\in G^{(0)}$.

A \emph{bisection} is a subset $S\subset G$ such that $r|_S$ and $s|_S$ are injective. Note that, if $S$ is open, then $r|_S$ and $s|_S$ are homeomorphisms onto their images. 
We will denote by $\mathcal{S}$ the inverse semigroup of open bisections of $G$, 
and by $\mathcal{C} \subset \mathcal{S}$ the sub-inverse semigroup of compact open bisections.
There is a homomorphism $\theta$ from $\mathcal{S}$ to the inverse semigroup of homeomorphisms between open subsets of $G^{(0)}$, 
given by $\theta_U:=r\circ (s|_U)^{-1}\colon s(U)\LRA r(U)$. 
As observed in~\cite{MR2460017}, $\theta$ is injective if and only if $G$ is \emph{essentially principal} (that is, $\Int\{g\in G:r(g)=s(g)\}=G^{(0)}$).

In the following we let $C_c(G)$ be the collection of complex valued, continuous and compactly supported functions on $G$.
This is a $^*$-algebra with the convolution product
\[
    f\star g(\gamma) = \sum_{\alpha \beta = \gamma} f(\alpha) g(\beta),
\]
for $f, g\in C_c(G)$ and $\gamma\in G$, and $^*$-involution $f^*(\gamma) = \overline{f(\gamma^{-1})}$, for $f\in C_c(G)$ and $\gamma\in G$.

Let $\mathrm C^*_r(G)$ and $\mathrm C^*(G)$ denote the reduced and full groupoid $\mathrm C^*$-algebras, respectively. 
For an introduction to (étale) groupoids and their $C^*$-algebras, the reader is refered to, e.g., \cite{putnam} or~\cite{Simsnotes}.

If $G$ is minimal and essentially principal, then $\mathrm C^*_r(G)$ is simple (see, e.g.,~\cite[Proposition 4.3.7]{Simsnotes}).

%As a consequence of \cite[Theorem 4.4]{MR2745642}, if $G$ is second countable, minimal and essentially principal, then $C^*_r(G)$ is simple. 

A regular Borel measure $\mu$ on $G^{(0)}$ is \emph{$G$-invariant} if $\mu(r(S))=\mu(s(S))$, for each $S\in\mathcal{S}$. 
Clearly, $\mu$ is $G$-invariant if and only if $\mu(r(U))=\mu(s(U))$, for each $U\in\mathcal{C}$. 
The following proposition is well-known.

\begin{proposition}\label{putnam}
Let $G$ be an ample groupoid with compact unit space. The following conditions are equivalent:
\begin{enumerate}
    \item[(i)] $G^{(0)}$ admits a $G$-invariant probability measure;
    \item[(ii)] $\mathrm C^*_r(G)$ admits a tracial state;
    \item[(iii)] $\mathrm C^*(G)$ admits a tracial state.
\end{enumerate}
\end{proposition} 

\begin{proof}
The proof of the implications (i)$\implies$(ii)$\implies$(iii) can be found in~\cite[Theorem 3.4.4]{putnam}.

(iii) $\implies$ (i):
Let $\tau$ be a tracial state on $\mathrm C^*(G)$. 
Given $U\in\mathcal{C}$, we have
\[
    \tau(1_{r(U)})=\tau(1_U1_{U^{-1}})=\tau(1_{U^{-1}}1_U)=\tau(1_{s(U)}).
\]
Thus, the probability measure on $G^{(0)}$ induced by $\tau|_{C(G^{(0)})}$ is $G$-invariant.
\end{proof}

Suppose $G^{(0)}$ admits a $G$-invariant measure $\mu$.
Then there is a representation $\rho\colon C_c(G)\LRA B(L^2(G^{(0)},\mu))$ given by 
\begin{equation}
    (\rho(f)(\xi))(x):=\sum_{g\in r^{-1}(x)}f(g)\xi(s(g)),\label{rho}
\end{equation}
for $f\in C_c(G)$, $\xi\in L^2(G^{(0)}, \mu)$ and $x\in G^{(0)}$. 

Note that $\rho|_{C(G^{(0)})}$ is the representation by multiplication operators. Moreover, if $U$ is a compact open bisection, then
\[
    (\rho(1_U)(\xi))(x)=
    \begin{cases}
        \xi(\theta_U^ {-1}(x)), & x\in s(U), \\
        0, & x\notin  s(U),
    \end{cases}
\]
for $\xi\in L^2(G^{(0)}, \mu)$ and $x\in G^{(0)}$.

\subsection{Topological full groups}

Given an ample groupoid $G$ with compact unit space, the \emph{topological full group} of $G$ is 
\[
    [[G]] := \{ U\in \mathcal{C}\mid r(U) = s(U) = G^{(0)} \}.
\]
This definition coincides with the one from~\cite{MR3904185}.
In~\cite{MR2876963}, however, H.~Matui defines the topological full group of $G$ as $\theta([[G]])$. 
Therefore, if $G$ is essentially principal then $\theta$ is injective and the two definitions coincide.

Two examples to have in mind are as follows.

\begin{example}\label{tg}
Let $\varphi$ be an action of a group $\Gamma$ on a compact Hausdorff space $X$. 
As a space, the \emph{transformation groupoid} associated with $\varphi$ is $G_\varphi:=\Gamma\times X$ equipped with the product topology.
The product of two elements $(h,y), (g,x)\in G_\varphi$ is defined if and only if $y=gx$ in which case $(h,gx)(g,x):=(hg,x)$.
Inversion is given by $(g,x)^{-1}:=(g^{-1},gx)$.
The unit space $G^{(0)}$ is naturally identified with $X$ and $G_\varphi$ is ample if $X$ is totally disconnected.

The topological full group $[[G_\varphi]]$ consists of sets of the form $\bigcup_{i=1}^n\{g_i\}\times A_i$, 
where $g_1,\dots,g_n\in\Gamma$ and $A_1,\dots,A_n\subset X$ are clopen sets such that 
\[
    X = \bigsqcup_{i=1}^n A_i = \bigsqcup_{i=1}^n g_iA_i.
\]
In particular, there is a canonical injective homomorphism $\Gamma\LRA [[G_\varphi]]$ sending $g\longmapsto \{g\}\times X$.
\end{example}

\begin{example}\label{full}
Let $X:=\{0,1\}^\mathbb{N}$ be the full one-sided $2$-shift and consider the Deaconu-Renault groupoid
\[
    G_{[2]}:= \{(y,n,x)\in X\times\mathbb{Z}\times X \mid \exists l,k\in \N: n = l-k, y_{l+i}=x_{k+i}~\forall i\in\mathbb{N}\}.
\]
The product of $(z,n,y'),(y,m,x)\in G_{[2]}$ is well-defined if and only if $y'=y$ in which case $(z,n,y)(y,m,x):=(z,n+m,x)$. 
Inversion is given by $(y,n,x)^{-1}:=(x,-n,y)$. 

Let $X_f$ be the set of finite words (including the empty word) on the alphabet $\{0,1\}$. 
Given $\alpha\in X_f$, let $|\alpha|$ denote its length and let $\overline{\alpha}:=\{x\in X \mid x_i=\alpha_i, 0\leq i < |\alpha|\}$ be the cylinder set of $\alpha$.
The topology on $G_{[2]}$ is generated by sets of the form 
\[
    Z(\beta,\alpha):= \{(y,|\beta|-|\alpha|,x)\in G_{[2]} \mid y\in\overline{\beta}, x\in\overline{\alpha}, y_{|\beta|+i}=x_{|\alpha|+i}~\forall i\in\mathbb{N}\},
\]
for $\alpha,\beta\in X_f$.
This topology is strictly finer than the one inherited from the product topology and $G_{[2]}$ is ample with compact unit space. Note as well that $G_{[2]}$ is minimal.

The topological full group $[[G_{[2]}]]$ consists of sets of the form 
\begin{equation}\label{top}
    \bigcup_{j=1}^n Z(\beta^j,\alpha^j),
\end{equation}
with $X = \bigsqcup_{j=1}^n \overline{\alpha^j} = \bigsqcup_{j=1}^n \overline{\beta^j}$.

We would now like to recall the isomorphism between Thompson's group $V$ and $[[G_{[2]}]]$, observed in~\cite{MR3377390} 
(see also~\cite{McT} and~\cite{MR2119267}).

Thompson's group $V$ consists of piecewise linear, right continuous bijections on $[0,1)$ which have finitely many points of non-differentiability, all being dyadic rationals, and have a derivative which is a power of $2$ at each point of differentiability.

Given $\alpha,\beta\in X_f$, let $\psi(\alpha):=\sum_i \alpha_i 2^{-i}\in[0,1)$ and $I(\alpha):=[\psi(\alpha),\psi(\alpha)+2^{-|\alpha|})$. 
The isomorphism from $[[G_{[2]}]]$ to $V$ takes $\bigcup_j Z(\beta^j,\alpha^j)$ as in \eqref{top} and sends it to the bijection on $[0,1)$ which, 
restricted to $I(\alpha^j)$, is linear, increasing and onto $I(\beta^j)$, for every $j$.
\end{example}

The next example shows that the short exact sequence induced by the quotient $\theta\colon [[G]]\LRA \theta({[[G]]})$ is not always split. 
Since we are interested in studying the canonical representation of $[[G]]$ in $\mathrm C^*(G)$, 
this illustrates why we have chosen to treat the topological full group as bisections, rather than homeomorphisms on the unit space.

\begin{example}
Let $X:=\mathbb{Z}\cup\{\infty\}$ be the one-point compactification of $\mathbb{Z}$
and define an action $\varphi\colon \Z\curvearrowright X$ by 
\[
    \varphi_n(x):=
    \begin{cases}
        {(-1)}^n x, & x\in \Z, \\
        \infty, & x = \infty,
    \end{cases}
\]
for $n\in\mathbb{Z}$.
Note that $\{1\}\times X$ is a compact open bisection in the transformation groupoid $G_\varphi$ and that the homeomorphism
\[
    \theta_{\{1\}\times X}(x) =
    \begin{cases}
        -x, & x\in \Z, \\
        \infty, & x = \infty,
    \end{cases}
\]
for $x\in X$, has order $2$. 

Moreover, for any $U\in[[G_\varphi]]$ satisfying $\theta_U=\theta_{\{1\}\times X}$, there is an odd integer $n$ such that $(n,\infty)\in U$. 
In particular, $U$ has infinite order. 
Therefore, the short exact sequence induced by $\theta\colon [[G_\varphi]]\LRA \theta([[G_\varphi]])$ is not split.
\end{example}

\subsection{Unitary representations}

Let $G$ be an an ample groupoid with compact unit space.
There is a unitary representation 
\begin{align*}
    \pi\colon [[G]]&\LRA \mathrm C^*(G)\\ 
    U&\longmapsto 1_U,
\end{align*}
We will denote the analogous representation of $[[G]]$ in $\mathrm C^*_r(G)$ by $\pi_r$.

If $\sigma$ and $\eta$ are unitary representations of a group $\Gamma$ on unital $\mathrm C^*$-algebras, then $\sigma$ is said to \emph{weakly contain} $\eta$ if 
\[
    \left\| \sum_i \alpha_i \eta(g_i) \right\| \leq \left\| \sum_i \alpha_i \sigma(g_i) \right\|,
\]
for every $\sum_i \alpha_i g_i\in \mathbb{C}\Gamma$.
The \emph{trivial representation} $\Gamma\LRA\mathbb{C}$ satisfies $g\longmapsto 1$, for every $g\in \Gamma$.

Given a unitary representation $\eta$ of $\Gamma$ on a unital $\mathrm C^*$-algebra $A$, 
we denote by $\mathrm C^*_\eta(\Gamma)$ the $\mathrm C^*$-algebra generated by the image of $\eta$. 
Note that if $\eta$ weakly contains the trivial representation,
then $\mathrm C^*_\eta(\Gamma)$ admits a character whose kernel is $\overline{\Span}\{1_A-\eta(g)\mid g\in\Gamma\}$.

\begin{proposition}\label{bonitinho}
    Let $\eta$ be a unitary representation of a group $\Gamma$ on a unital $\mathrm C^*$-algebra $A$. 
    Then $\eta$ weakly contains the trivial representation if and only if $1_A\notin\overline{\Span}\{1_A - \eta(g)\mid g\in \Gamma\}$.
\end{proposition}

\begin{proof}
The forward implication is evident, so we only prove the backward one.

Let $B:=\overline{\Span}\{1_A-\eta(g):g\in \Gamma\}$. If $1_A\notin B$, then, since $B$ is a $\mathrm C^*$-algebra, $\mathrm{dist}(1_A,B)=1$.
Hence, for every $\alpha_1,\dots,\alpha_n\in\mathbb{C}$ and $g_1,\dots,g_n\in \Gamma$, we have that
\[
    \left\|\sum\alpha_i\eta(g_i)\right\| = \left\|\left(\sum\alpha_i\right).1_A-\sum\alpha_i(1_A-\eta(g_i))\right\| \geq \left|\sum\alpha_i\right|,
\]
thus showing that $\eta$ weakly contains the trivial representation.
\end{proof}

%%%%%%%%%%%%%%%%%%%%%%%%%%%%%%%%%%%%%%%%%%%%%%%%%%%%%%%%%%%%%%%%%%%%%
%%%%%%%%%%%%%%%%%%%%%%%%%%%%%%%%%%%%%%%%%%%%%%%%%%%%%%%%%%%%%%%%%%%%%
\section{Groups covering groupoids}\label{cover}

An ample groupoid $G$ can always be covered by compact open bisections. 
We investigate to which degree $G$ can be covered by compact open bisections $U$ which satisfy $r(U)=s(U)=G^{(0)}$.
We show that if $\Gamma\leq [[G]]$ covers $G$ and $\mu$ is a $\Gamma$-invariant probability measure on $G^{(0)}$,
then $\mu$ is also $G$-invariant.

\begin{definition}
    Given an ample groupoid $G$ with compact unit space, we say that a subgroup $\Gamma\leq [[G]]$ \emph{covers} $G$ if $G=\bigcup_{U\in\Gamma}U$. 
\end{definition}

The idea of covering a groupoid $G$ by compact open bisections $U$ such that $r(U)=s(U)=G^{(0)}$ has already appeared in 
H.~Matui's study of automorphisms of $G$, cf.~\cite[Proposition 5.7]{MR2876963}.

If $G$ is essentially principal, then a subgroup $\Gamma\leq [[G]]$ covers $G$ if and only if, 
for each open bisection $S$ and $x\in s(S)$, there are $U\in\Gamma$ and a neighborhood $W\subset s(S)$ of $x$ such that $\theta_U|_W=\theta_S|_W$. 

\begin{example}
If $\varphi$ is an action of a group $\Gamma$ on a compact Hausdorff and totally disconnected space, then the copy of $\Gamma$ in $[[G_\varphi]]$ covers $G_\varphi$.
\end{example}

\begin{example}
Recall that Thompson's group $T< V$ consists of the elements of Thompson's group $V$ (see Example~\ref{full}) which have at most one point of discontinuity.

Let $G_{[2]}$ be the groupoid of Example~\ref{full}. 
Under the identification of $V$ with $[[G_{[2]}]]$, $T$ covers $G_{[2]}$. 
This follows from the fact that if $I,J\subset[0,1)$ are left-closed and right-open intervals with endpoints in $\mathbb{Z}[1/2]$, 
then there exists a piecewise linear homeomorphism $f\colon I\LRA J$ with a derivative which is a power of $2$ at each point of differentiability
and with finitely many points of non-differentiability, all of which belong to $\mathbb{Z}[1/2]$.
\end{example}

\begin{lemma}\label{cuida}
Let $G$ be an ample groupoid with compact unit space.
If $|G(x)|\geq 2$ for every $x\in G^{(0)}$, then $[[G]]$ covers $G$.
\end{lemma}

\begin{proof}
Let $g\in G$. 
If $r(g)\neq s(g)$, then there is a compact open bisection $V$ containing $g$ and such that $s(V)\cap r(V)=\emptyset$. 
Let $U := V\cup V^{-1}\cup (G^{(0)}\setminus(s(V)\cup r(V)))$. 
Then $g\in U\in[[G]]$.

If $r(g)=s(g)$, then there is $h\in s^{-1}(r(g))$ such that $r(hg)=r(h)\neq s(h)=s(hg)$ since $|G(r(g))|\geq 2$. 
As before, there are $U,U'\in[[G ]]$ such that $h\in U$ and $hg\in U'$. 
Hence, $g\in U^{-1}U'\in [[G]]$. 
\end{proof}

The purpose of the next example is to show that the above result may fail if one does not make any assumption on the orbits. 

\begin{example}\label{cor}
Consider $X:=\mathbb{Z}\cup\{\pm\infty\}$ equipped with the order topology and let $\varphi\colon \Z\curvearrowright X$ be the action given by $\varphi_t(x):=t+x$, 
for $t\in\Z$ and $x\in X$.
The transformation groupoid $G_\varphi$ is ample with compact unit space.

Given $x,z\in X$ we put $[x,z] := \{y\in X:x\leq y \leq z\}$.
Then 
\[
    H := \{(t,x)\in \mathbb{Z}\times [0,+\infty]: -t\leq x\}
\]
is an ample subgroupoid of $G_\varphi$. 
Incidentally, this is the groupoid of the partial action obtained by restricting $\varphi$ to $[0,+\infty]$ (see \cite{MR3699795} and \cite{MR3694599} for more details).
Observe that $|H((0,+\infty))| = 1$.

We claim that if $U\in[[H]]$, then $(1,+\infty)\notin U$. 
Otherwise, there is $t\in\mathbb{N}$ such that $S:=\{1\}\times[t,+\infty]\subset U$ and $U\setminus S\in\mathcal{C}$.
But then $s(U\setminus S)=[0,t-1]$ and $r(U\setminus S)=[0,t]$ contradicting the fact that $r$ and $s$ are injective on $U\setminus S$. 
Hence, $(1,+\infty)\notin U$ and $[[H]]$ does not cover $H$.
\end{example}

Recall that a probability measure $\mu$ on $G^{(0)}$ is \emph{$G$-invariant} if $\mu(s(S)) = \mu(r(S))$ for every $S\in \mathcal{S}$.
Moreover, if $\Gamma\leq [ [G]]$, then we say $\mu$ is \emph{$\Gamma$-invariant} if it is invariant with respect to the action $\theta$.

\begin{proposition}\label{surpre}
Let $G$ be an ample groupoid with compact unit space and $\Gamma$ a subgroup of $[[G]]$. Consider the following conditions:

\begin{enumerate}
\item[(i)] $G^{(0)}$ admits a $G$-invariant probability measure;
\item[(ii)]$\pi|_\Gamma$ weakly contains the trivial representation;
\item[(iii)] $C^*_\pi(\Gamma)$ admits a character;
\item[(iv)] $G^{(0)}$ admits a $\Gamma$-invariant probability measure.
\end{enumerate}
Then (i)$\implies$(ii)$\implies$(iii)$\implies$(iv). 
If $\Gamma$ covers $G$, then (iv) $\implies$ (i) and all conditions are equivalent.
\end{proposition} 

\begin{proof}
(i)$\implies$(ii): 
Suppose $\mu$ is a $G$-invariant measure on $G^{(0)}$ and let $\rho\colon C_c(G) \LRA B(L^2(G^{(0)},\mu))$ be the representation given by~\eqref{rho}.
The vector $1_{G^{(0)}}\in L^2(G^{(0)},\mu)$ is invariant for the representation $\rho\circ \pi|_\Gamma\colon \Gamma\LRA B(L^2(G^{(0)},\mu))$. 
Hence, $\pi|_\Gamma$ weakly contains the trivial representation.

The implication (ii)$\implies$(iii) is evident.

(iii)$\implies$(iv): 
Let $\varphi$ be a character on $\mathrm C^*_\pi(\Gamma)$ and $\tau$ a state on $\mathrm C^*(G)$ which is an extension of $\varphi$. 
Then $\mathrm C^*_\pi(\Gamma)$ is in the multiplicative domain of $\tau$. 
Clearly, $\tau|_{C(G^{(0)})}$ induces a $\Gamma$-invariant probability measure on $G^{(0)}$.

Now, suppose $\Gamma$ covers $[[G]]$ and let us show that (iv) $\implies$ (i). 
Let $\mu$ be a $\Gamma$-invariant probability measure on $G^{(0)}$.
We claim that $\mu$ is also $G$-invariant.
Indeed, since $\Gamma$ covers $G$, given $S\in\mathcal{C}$, we have that $S=\bigcup_{U\in\Gamma}(S\cap U)$. 
As $S$ is compact, there are $S_1,\dots,S_n\in\mathcal{C}$ and $U_1,\dots, U_n\in\Gamma$ such that $S=\bigsqcup_i S_i$ and $S_i\subset U_i$ for $1\leq i \leq n$. 
In particular, $\theta_{U_i}(s(S_i))=r(S_i)$ for every $i$. 
It follows that
\[
    \mu(r(S))=\sum_{i=1}^n \mu(r(S_i)) = \sum_{i=1}^n \mu(s(S_i))=\mu(s(S)). 
\]
Therefore, $\mu$ is a $G$-invariant probability measure on $G^{(0)}$.
\end{proof}

\begin{corollary}\label{nice}
Let $G$ be an ample groupoid with compact unit space. The following conditions are equivalent:
\begin{enumerate}
    \item[(i)] $G^{(0)}$ admits a $G$-invariant probability measure;
    \item[(ii)] $\pi$ weakly contains the trivial representation;
    \item[(iii)] $\mathrm C^*_\pi([[G]])$ admits a character;
    \item[(iv)] $G^{(0)}$ admits a $[[G]]$-invariant probability measure.
\end{enumerate}
\end{corollary}

\begin{proof}
The implications (i)$\implies$(ii)$\implies$(iii)$\implies$(iv) follow from Proposition~\ref{surpre}.

(iv)$\implies$(i): 
If, for each $x\in G^{(0)}$, $|G(x)|\geq 2$, then the result follows from Lemma~\ref{cuida} and Proposition~\ref{surpre}.

If there is $x\in X$ such that $|G(x)|=1$, then point evaluation at $x$ is a $G$-invariant probability measure. 
\end{proof}

%%%%%%%%%%%%%%%%%%%%%%%%%%%%%%%%%%%%%%%%%%%%%%%%%%%%%%%%%%%%%%%%%%%%%%%%%%%%%%%%%%
%%%%%%%%%%%%%%%%%%%%%%%%%%%%%%%%%%%%%%%%%%%%%%%%%%%%%%%%%%%%%%%%%%%%%%%%%%%%%%%%%%
\section{Representations of topological full groups}\label{mains}

In this section, we prove the main results of the article. We start with two technical lemmas.

\begin{lemma}\label{tec}
Let $G$ be an ample groupoid with compact unit space. 
If $S,T\in[[G]]$ and $W\subset G^{(0)}$ is a clopen subset such that $\theta_S(W), W, \theta_T^{-1}(W)$ are mutually disjoint, 
then $(1-1_S)1_W(1-1_T)\in\Span\{1-1_U\in C_c(G) \mid U\in[[G]]\}$.
\end{lemma}

\begin{proof}
We have
\begin{align*}
    (1-1_S)1_W(1-1_T) &= 1_{SWT}+1_{T^{-1}WS^{-1}}+1_W+1_{G^{(0)}\setminus(\theta_S(W) \cup W\cup\theta_T^{-1}(W))}\\
    & -(1_{T^{-1}WS^{-1}}+1_{G^{(0)}\setminus(\theta_S(W)\cup W\cup\theta_T^{-1}(W))}+1_{SW}+1_{WT}).
\end{align*}
The sets $SWT, T^{-1}WS^{-1}, W$ and $G^{(0)}\setminus(\theta_T^{-1}(W)\cup W\cup\theta_S(W))$ are mutually disjoint
and their union is in $[[G]]$.
This is also the case for the sets $T^{-1}WS^{-1}, SW, WT$ and $G^{(0)}\setminus(\theta_S(W)\cup W\cup\theta_T^{-1}(W))$ and so the result follows.
\end{proof}

In order to employ Lemma~\ref{tec}, the following result will be useful.

\begin{lemma}\label{duv}
Let $G$ be an ample groupoid with compact unit space. 
If $x\in G^{(0)}$ and $y\in G(x)\setminus\{x\}$, then 
\begin{align}
    \Span\{1-1_U \mid U\in[[G]]\} 
    &= \Span\{1_L(1-1_S) \mid S,L\in[[G]],\theta_S(x)=y\}\label{ach1}\\
    &=\Span\{(1-1_T)1_R \mid T,R\in[[G]],\theta_T^{-1}(x)=y\}.\label{ach2}
\end{align}
\end{lemma}

\begin{proof}
Let 
\[
    B:=\Span\{1_L(1-1_S) \mid S,L\in[[G]],\theta_S(x)=y\}
\]
and take $U\in[[G]]$.
We will show that $1-1_U\in B$.

If $\theta_U(x)=x$, we take $L\in[[G]]$ such that $\theta_{LU}(x)=\theta_L(x)=y$. 
Then $1-1_U=1_{L^{-1}}(1_L-1)+1_{L^{-1}}(1-1_{LU})\in B$. 

On the other hand, if $\theta_U(x)\neq x$, we take $L\in[[G]]$ such that $\theta_L(x)=x$ and $\theta_{LU}(x)=y$. 
Then $\theta_{L^{-1}}(x) = x$ so $1-1_{L^{-1}}\in B$ by the above.
Hence $1-1_U = (1- 1_{L^{-1}}) + 1_{L^{-1}}(1-1_{LU})\in B$ proving~\eqref{ach1}. 

By taking adjoints and interchanging $x$ and $y$, the equality in \eqref{ach2} follows from \eqref{ach1}.
\end{proof}

The next result generalizes~\cite[Theorem 3.7]{scarparo2017c}, which was obtained in the setting of Cantor minimal $\mathbb{Z}$-systems.

\begin{theorem}\label{imp}
Let $G$ be an ample groupoid with compact unit space.
If $|G(x)|\geq 3$ for every $x\in G^{(0)}$, then $\overline{\Span}\{1-1_U\in \mathrm C^*(G) \mid U\in[[G]]\}$ is a hereditary $\mathrm C^*$-subalgebra of $\mathrm C^*(G)$.
\end{theorem}

\begin{proof}
Let $B:=\overline{\Span}\{1-1_U\in \mathrm C^*(G) \mid U\in[[G]]\}$. We will first show that
\begin{equation}\label{conti}
    B C(G^{(0)}) B\subset B.
\end{equation}
It suffices to prove that, given $U,V\in[[G]]$, there is a basis $\mathcal{W}$ for $G^{(0)}$ consisting of compact open sets satisfying $(1-1_U)1_W(1-1_V)\in B$, for each $W\in\mathcal{W}$.
Take $x\in G^{(0)}$ and let $y$ and $z$ be distinct elements in $G(x)\setminus\{x\}$. 
By Lemma~\ref{duv}, there are $n\in \N$ and $L_1,\ldots,L_n$ ,$U_1,\ldots,U_n$, $V_1,\ldots,V_n$ and $R_1,\ldots,R_n$ in $[ [G]]$ and $\alpha_1,\ldots,\alpha_n$, $\beta_1,\ldots,\beta_n\in \C$ such that
\[
    1-1_U = \sum_{i=1}^n \alpha_i1_{L_i}(1-U_i), \qquad 1 -1_V = \sum_{i=1}^n \beta_i(1-1_{V_i})R_i
\]
with $\theta_{U_i}(x)=y$ and $\theta_{V_i}^{-1}(x)=z$ for every $i=1,\ldots,n$. 
By Lemma~\ref{tec}, we see that $(1-1_U)1_W(1-1_V)\in B$ for every sufficiently small compact open neighborhood $W$ of $x$. 
This proves \eqref{conti}.

Next we show that $B \mathrm C^*(G) B\subset B$. 
It suffices to prove that $B 1_W B\subset B$, for every $W$ in a basis for $G$ consisting of compact open sets. Given $g\in G$, take $U\in[[G]]$ such that $\theta_U(r(g))\neq s(g)$. 
Then, for $W\subset G^{(0)}$ sufficiently small compact open neighborhood of $g$, we have that $\theta_U(r(W))\cap s(W)=\emptyset$.
Let 
\[
    V:=UW\cup (UW)^{-1}\cup(G^{(0)}\setminus(\theta_U(r(W))\cup s(W)))\in[[G]].
\]
Since $\theta_U(r(W))\cap s(W)=\emptyset$, we have $UWV = \theta_U(r(W))\subset G^{(0)}$ and, finally,
\[
    B 1_W B = B (1_U 1_W 1_V) B = B 1_{\theta_U(r(W))} B \subset B
\]
by~\eqref{conti}.
\end{proof}

\begin{corollary}\label{cor:imp-reduced}
Let $G$ be an ample groupoid with compact unit space.
If $|G(x)|\geq 3$ for every $x\in G^{(0)}$, 
then $\overline{\Span}\{1-1_U\in \mathrm C_r^*(G) \mid U\in[[G]]\}$ is a hereditary $\mathrm C^*$-subalgebra of $\mathrm C_r^*(G)$.
If, in addition, $G$ is second countable, essentially principal and minimal, 
then $\overline{\Span}\{1-1_U\in \mathrm C_r^*(G) \mid U\in[[G]]\}$ is stably isomorphic to $\mathrm C_r^*(G)$.
\end{corollary}

\begin{proof}
    The first assertion follows directly from Theorem~\ref{imp} while the second follows from simplicity of $\mathrm C_r^*(G)$
    and Brown's theorem~\cite[Theorem 2.8]{MR0454645}.
\end{proof}

The next example shows that Theorem~\ref{imp} does not hold without the hypothesis on orbits.

\begin{example}
Let $X:=\mathbb{Z}\cup\{\pm\infty\}$ with the order topology and 
let $\varphi$ be the action of the infinite dihedral group $\mathbb{Z}\rtimes\mathbb{Z}_2$ on $X$ given by $\varphi_{(n,j)}(x):=n+{(-1)}^{j}x$, 
for $(n,j)\in\mathbb{Z}\rtimes\mathbb{Z}_2$ and $x\in X$.   
Then $|G_\varphi(x)|\geq 2$ for every $x\in G_\varphi^{(0)}$.

By arguing as in Example~\ref{cor}, one concludes that, given $U\in [[G_\varphi]]$, there is $(n,j)\in\mathbb{Z}\rtimes\mathbb{Z}_2$ such that $((n,j),\pm\infty)\in U$.

Let $E\colon \mathrm C^*(G_\varphi)\LRA C(G_\varphi^{(0)})$ be the canonical conditional expectation and 
let $\delta_{+\infty}$ and $\delta_{-\infty}$ be the two states on $C(G_\varphi^{(0)})$ given by point-evaluations at $+\infty$ and $-\infty$, respectively. 
Then $\delta_{+\infty}\circ E$ and $\delta_{-\infty}\circ E$ are two distinct states on $\mathrm C^*(G_\varphi)$
whose restrictions to $B=\overline{\Span}\{1-1_U\in \mathrm C^*(G_\varphi)\mid U\in[[G_\varphi]]\}$ agree.
Hence, $B$ is not a hereditary $\mathrm C^*$-subalgebra of $\mathrm C^*(G_\varphi)$.
\end{example}

By combining Theorem~\ref{imp} with the results of the previous section, we obtain the following:

\begin{theorem}\label{main}
Let $G$ be an ample groupoid with compact unit space.
Assume that $|G(x)|\geq 3$ for every $x\in G^{(0)}$. 
The following conditions are equivalent.
\begin{enumerate}
    \item[(i)] $\mathrm C^*(G)$ admits no tracial state;
    \item[(ii)] $\mathrm C^*_\pi([[G]])$ admits no character;
    \item[(iii)] $\pi$ does not weakly contain the trivial representation;
    \item[(iv)] $\mathrm C^*_\pi([[G]]) = \mathrm C^*(G)$.
\end{enumerate}
\end{theorem}

\begin{proof}
The equivalences (i)$\iff$(ii)$\iff$(iii) follow from Proposition~\ref{putnam} and Corollary~\ref{nice}.

(iii)$\implies$(iv):
By Proposition~\ref{bonitinho} and Theorem~\ref{imp}, 
$B:=\overline{\Span}\{1-1_U \mid U\in[[G]]\}$ is a hereditary $\mathrm C^*$-subalgebra of $\mathrm C^*(G)$ and $1_{\mathrm C^*(G)}\in B$. Hence, $B = \mathrm C^*(G)$.
Since $B\subset \mathrm C^*_\pi([[G]])$ the result follows.

(iv)$\implies$(i):
If $\mathrm C^*(G)$ has a tracial state, then $G^{(0)}$ admits an invariant probability measure $\mu$, cf.~Proposition~\ref{putnam}. 
Since $|G(x)|>1$ for each $x\in G^{(0)}$, $\mu$ cannot be a point-evaluation.
Let $\rho$ be the representation of $C_c(G)$ in $B(L^2(G^{(0)},\mu))$ as in~\eqref{rho}. 
Then $\rho$ extends to a representation of $\mathrm C^*(G)$ and of $\mathrm C^*_\pi([ [G]])$.
Note that the vector $1_{G^{(0)}}\in L^2(G^{(0)},\mu)$ is invariant under $\rho(\pi([ [G]]))$ and thus under $\rho|_{\mathrm C^*_\pi([ [G]])}$.
Now, if $\mathrm C^*_\pi([ [G]]) = \mathrm C^*(G)$, then $C(G^{(0)}) \subset \mathrm C^*_\pi([ [G]])$ 
but $\mathbb{C}1_{G^{(0)}}$ is not invariant under $\rho|_{C(G^{(0)})}$.
Indeed, if $X\subset G^{(0)}$ is any proper, non-empty subset which is compact and open, then $\rho(1_X)(1_{G^{(0)}}) = 1_X$.
Therefore $\mathrm C^*_\pi([ [G]])\neq \mathrm C^*(G)$.
\end{proof}

\begin{remark}\label{ho}
In~\cite[Proposition 5.3]{MR3630641}, U.~Haagerup and K.~Olesen considered a certain representation $\sigma$ of Thompson's group $V$ in the Cuntz algebra $\mathcal{O}_2$ 
and showed that $\mathrm C^*_\sigma(V)=\mathcal{O}_2$.
Under the identifications of $V$ with $[[G_{[2]}]]$ (see Example~\ref{full}) and $\mathcal{O}_2$ with $\mathrm C^*(G_{[2]})$, 
one can check that $\sigma$ and $\pi$ coincide. 
Hence, Theorem~\ref{main} recovers part of U.~Haagerup and K.~Olesen's result.
\end{remark}

We now state and prove a version of Theorem~\ref{main} regarding $\mathrm C^*_r(G)$. 

\begin{theorem}\label{surp}
Let $G$ be an ample groupoid with compact unit space.
Assume that $|G(x)|\geq 3$ for each $x\in G^{(0)}$ and consider the following conditions:
\begin{enumerate}
    \item[(i)] $\mathrm C^*_{r}(G)$ admits no tracial state;
    \item[(ii)] $\mathrm C^*_{\pi_r}([[G]])$ admits no character;
    \item[(iii)] $\pi_r$ does not weakly contain the trivial representation;
    \item[(iv)] $\mathrm C^*_{\pi_r}([[G]]) = \mathrm C^*_r(G)$.
\end{enumerate}
Then \textup{(i) $\implies$ (ii) $\iff$ (iii) $\iff$ (iv)}.
\end{theorem}

\begin{proof}
The implications (i)$\implies$(ii)$\implies$(iii)$\implies$(iv) are done as in the full case.

(iv)$\implies$(ii). 
If $\mathrm C^*_{\pi_r}(G) = \mathrm C^*_r(G)$ admits a character $\tau$, then $\tau|_{C(G^{(0)})}$ is a point evaluation at some $x\in G^{(0)}$.
As $\tau$ is a tracial state, it follows that for each compact and open bisection $S$ with $x\in s(S)$, we have $\theta_S(x)=x$. 
This contradicts the hypothesis that $|G(x)|>1$.
\end{proof}

The next example shows that the implication from (ii) to (i) in the above theorem fails in general, 
even in the case when $G$ is a principal, minimal and ample groupoid with unit space homeomorphic to the Cantor set.

\begin{example}
Let $\Gamma$ be a non-amenable, countable and residually finite group. 
There is a descending sequence ${(\Gamma_n)}_n$ of finite-index normal subgroups of $\Gamma$ 
such that the canonical map $j\colon\Gamma\LRA\prod\frac{\Gamma}{\Gamma_n}$ is injective. 
Then $X:=\overline{j(\Gamma)}$ is a topological group homeomorphic to the Cantor set. 
Furthermore, the action $\varphi$ by multiplication of $\Gamma$ on $X$ is free, minimal 
and the Haar measure on $X$ is $\Gamma$-invariant (actions of this sort were studied in detail in \cite{MR2427048}).

Then $\mathrm C^*_r(G_\varphi)$ admits a tracial state, whereas $\mathrm C^*_{\pi_r}([[G_\varphi]])$ does not admit a character, 
since $\mathrm C^*_r(\Gamma)$ embeds unitally in it and $\Gamma$ is non-amenable.
\end{example}

%%%%%%%%%%%%%%%%%%%%%%%%%%%%%%%%%%%%%%%%%%%%%%%%%%%%%%%%%%%%%%%%%%%%%%%%%%%%%%%%%%%%%%%%%
%%%%%%%%%%%%%%%%%%%%%%%%%%%%%%%%%%%%%%%%%%%%%%%%%%%%%%%%%%%%%%%%%%%%%%%%%%%%%%%%%%%%%%%%%
\section{$\mathrm C^*$-simplicity of topological full groups}\label{simp}

As an application of the above results, we provide conditions which ensure that the topological full group of an ample groupoid is $\mathrm C^*$-simple.

Recall that an ample groupoid $G$ is \emph{amenable} if there exists a net $(\mu_i)_i$ in $C_c(G)$ of non-negative functions such that
\begin{equation}\label{an}
    \sum_{h\in s^{-1}(r(g))}\mu_i(h)\LRA 1 \quad 
    \text{and} \quad
    \sum_{h\in s^{-1}(r(g))}|\mu_i(h)-\mu_i(hg)|\LRA 0,
\end{equation}
for $g\in G$, uniformly on compact subsets of $G$.
Amenability of $G$ is equivalent to nuclearity of $\mathrm C^*_r(G)$, and it implies that $\mathrm C^*(G)$ and $\mathrm C^*_r(G)$ are canonically isomorphic. 
For a proof of these facts, see, e.g., \cite{zbMATH05256855} and \cite{MR2536186}.
R.~Willett constructed in \cite{Willett} an example of non-amenable groupoid $G$ such that $\mathrm C^*(G)$ is canonically isomorphic to $\mathrm C^*_r(G)$.

\begin{lemma}\label{amen}
Let $G$ be an ample groupoid with compact unit space. 
If $\Gamma\leq [[G]]$ is an amenable subgroup which covers $G$, then $G$ is amenable.
\end{lemma}

\begin{proof}
We are going to construct functions satisfying~\eqref{an}. 
Let $K\subset G$ be a compact subset and let $\epsilon>0$.
As $\Gamma$ covers $G$ and is amenable, there are $V_1,\dots,V_n\in\Gamma$ such that $K\subset\bigcup_{i=1}^n V_i$ and a finite subset $F\subset \Gamma$ such that
\[
    \frac{|F\triangle F V_i|}{|F|} < \epsilon,
\]
for $1\leq i \leq n$.
Let $\mu: = \frac{1}{|F|}\sum_{U\in F} 1_U$. 
For $x\in G^{(0)}$ we have $\sum_{h\in s^{-1}(x)}\mu(h)=1$. 
Given $g\in K$, take $V_i$ such that $g\in V_i$. 
Then, for $h\in s^{-1}(r(g))$,
\begin{align*}
    |F|| \mu(h) - \mu(hg)| 
    &= \left|\sum_{U\in F}1_U(h)-1_{UV_i^{-1}}(h)\right|\\
    &\leq\left|\sum_{U\in F\setminus(FV_i^{-1})}1_U(h)-\sum_{U\in F\setminus(FV_i)}1_{UV_i^{-1}}(h)\right| \\
    &+\left|\sum_{U\in F\cap FV_i^{-1}}1_U(h)-\sum_{U\in F\cap FV_i}1_{UV_i^{-1}}(h)\right|\\
    &\leq\sum_{U\in F\setminus(FV_i^{-1})}1_U(h)+\sum_{U\in F\setminus(FV_i)}1_{UV_i^{-1}}(h).
\end{align*}

Consequently,
\begin{align*}
    \sum_{h\in s^{-1}(r(g))}|\mu(h)-\mu(hg)|
    &\leq\frac{1}{|F|}\sum_h\left(\sum_{U\in F\setminus(FV_i^{-1})}1_U(h)+\sum_{U\in F\setminus(FV_i)}1_{UV_i^{-1}}(h)\right)\\
    &=\frac{|F\setminus (FV_i^{-1})|+|F\setminus (FV_i)|}{|F|}
    <\epsilon.
\end{align*}
Therefore, $G$ is amenable.
\end{proof}

The converse implication is not true in general, see, e.g., Remark~\ref{rem:Elek-Monod}.

Suppose a group $\Gamma$ is acting on a set $X$ and let $U\subset X$ be a subset.
The \emph{rigid stabilizer} of $U$ with respect to the action is the subgroup $\Gamma_U \leq \Gamma$ of the elements which pointwise fix the complement $X\setminus U$.
Let $W \subset G^{(0)}$ be non-empty and clopen and let $G_W = r^{-1}(W) \cap s^{-1}(W)$ be the restricted groupoid.
If ${[[G]]}_W$ is the rigid stabilizer of $W$ with respect to the action $\theta\colon [[G]]\curvearrowright G^{(0)}$,
then there is a surjective homomorphism ${[[G]]}_W \LRA [[G_W]]$ given by restriction.
If $G$ is essentially principal this map is an isomorphism.

\begin{theorem}\label{thm:nonamenable-C-simple}
Let $G$ be a second countable, essentially principal, minimal and ample groupoid with compact unit space. 
If 
\begin{enumerate}
\item[(i)]$G$ is not amenable, or
\item[(ii)]$\pi_r$ does not weakly contain the trivial representation,
\end{enumerate}
 then $[[G]]$ is $\mathrm C^*$-simple.
\end{theorem}

\begin{proof}
Assume $[[G]]$ is not $\mathrm C^*$-simple.
By~\cite[Theorem 3.7]{boudec2016subgroup}, there exists a non-empty and clopen $W\subset G^{(0)}$ such that the rigid stabilizer ${[[G]]}_W\cong[[G_W]]$ is amenable. 
Clearly, $G_W$ is an essentially principal, minimal and ample groupoid with compact unit space. 
By Lemma~\ref{amen}, $G_W$ is thus amenable and $\mathrm C^*_r(G_W)$ is nuclear. 

Since $\mathrm C^*_r(G)$ is simple, the projection $1_W\in \mathrm C^*_r(G)$ is full.
It therefore follows from~\cite[Lemma 5.2]{MR2876963} and Brown's theorem (\cite[Theorem 2.8]{MR0454645}) 
that the full corner $\mathrm C^*_r(G_W) = 1_W \mathrm C_r^*(G) 1_W $ is stably isomorphic to $\mathrm C^*_r(G)$. 
Consequently, $\mathrm C^*_r(G)$ is nuclear, and $G$ is amenable. 

Furthermore, amenability of $[[G_W]]$ implies that $W$ admits a $[[G_W]]$-invariant probability measure. 
Corollary~\ref{nice} and Proposition~\ref{putnam} then imply that $\mathrm C^*_r(G_W)$ admits a tracial state. 
As $\mathrm C^*_r(G_W)$ is simple, the tracial state is faithful. 
Hence, $\mathrm C^*_r(G_W)$ is stably finite and, consequently, so is $\mathrm C^*_r(G)$.

Now,~\cite[Theorem 6.5]{rainone2017dichotomy} (or~\cite[Theorem 5.14]{bonicke2017ideal}) implies that $\mathrm C^*_r(G)$ admits a tracial state. 
Since $\mathrm C^*_r(G) = \mathrm C^*(G)$, we conclude from Corollary~\ref{nice} again that $\pi=\pi_r$ weakly contains the trivial representation.
\end{proof}

The next corollary is an immediate consequence of Theorems~\ref{surp} and~\ref{thm:nonamenable-C-simple}.

\begin{corollary}
    Let $G$ be a second countable, essentially principal, minimal and ample groupoid with compact unit space. 
    If $\mathrm C^*_{r}(G)$ admits no tracial state, then $[[G]]$ is $\mathrm C^*$-simple.
\end{corollary}

Recall that an action of a group $\Gamma$ on a topological space $X$ is \emph{topologically free}
if $\Int\{x\in X \mid gx = x\}= \emptyset$, for each $g\in \Gamma\setminus\{e\}$. 
The following result generalizes~\cite[Theorem 4.38]{boudec2016subgroup}, which assumed freeness of the action. 

\begin{corollary}\label{cor:action-C-simple}
    Let $\varphi$ be a topologically free and minimal action of a countable and non-amenable group $\Gamma$ on the Cantor set. 
    Then $[[G_\varphi]]$ is $\mathrm C^*$-simple.
\end{corollary}

\begin{proof}
Since $\mathrm C^*_r(\Gamma)$ embeds unitally in $\mathrm C^*_{\pi_r}([[G_\varphi]])$, 
non-amenability of $\Gamma$ implies that $\pi_r$ does not weakly contain the trivial representation.
\end{proof}

\begin{remark}\label{rem:Elek-Monod}
In~\cite{MR3080176}, G.~Elek and N.~Monod constructed a free and minimal action $\varphi$ of $\mathbb{Z}^2$ on the Cantor set such that $[[G_\varphi]]$ is not amenable. 
This example is not covered by Theorem~\ref{thm:nonamenable-C-simple}, and we do not know whether $[[G_\varphi]]$ is $\mathrm C^*$-simple.
\end{remark}

\bibliographystyle{acm}
\bibliography{bibliografia}
\end{document}